\documentclass{article}

\usepackage{enumerate}

\usepackage[centertags]{amsmath}
\usepackage{amsfonts}
\usepackage{amssymb}
\usepackage{amsthm}
\usepackage{newlfont}
\usepackage{mathtools}
\usepackage[top=1in, bottom=1in, left=1in, right=1in]{geometry}
\usepackage{tikz}

\theoremstyle{plain}
\newtheorem{theorem}{Theorem}[section]
\newtheorem{proposition}[theorem]{Proposition}
\newtheorem{corollary}[theorem]{Corollary}
\newtheorem{lemma}[theorem]{Lemma}

\theoremstyle{definition}
\newtheorem{definition}[theorem]{Definition}

\newtheorem*{remark}{Remark}

\newcommand{\R}{\mathbb{R}}
\newcommand{\N}{\mathbb{N}}

\newcommand{\Deg}{\operatorname{Deg}}
\newcommand{\D}{\Deg_{\max}}

\newcommand{\eps}{\varepsilon}

\newcommand{\Hidden}[1]{}

\newcommand{\cC}{\mathcal{C}}

\newcommand{\drm}{\mathrm{d}}
\newcommand{\euler}{\mathrm{e}}
\newcommand{\cL}{\mathcal{L}}

\DeclareMathOperator{\vol}{vol}

\DeclareMathOperator{\diam}{diam}

\begin{document}

\title{Spectrally positive Bakry-{\'E}mery Ricci curvature on graphs}
\author{Florentin M\"unch\footnote{MPI Leipzig, muench@mis.mpg.de}~~~~~Christian Rose\footnote{MPI Leipzig, crose@mis.mpg.de}}
\date{}
\maketitle

\begin{abstract}
We investigate analytic and geometric implications of non-constant Ricci curvature bounds. We prove a Lichnerowicz eigenvalue estimate and finiteness of the fundamental group assuming that $\mathcal L+2\operatorname{Ric}$ is a positive operator where $\mathcal L$ is the graph Laplacian.
Assuming that the negative part of the Ricci curvature is small in Kato sense, we prove diameter bounds, elliptic Harnack inequality and Buser inequality.
This article seems to be  the first one establishing these results while allowing for some negative curvature.

\end{abstract}

%
%
%
%
%
%
%
%
%


\section{Introduction}

Analytic properties of graphs and manifolds such as eigenvalue estimates for the Laplacian or gradient estimates for the heat semigroup depend crucially on curvature conditions. 
In the last decade, there has been increasing interest in discrete Ricci curvature notions based on
\cite{schmuckenschlager1999curvature,
ollivier2007ricci,erbar2012ricci,
mielke2013geodesic}. Indeed, most of the research so far is focused on constant curvature bounds.

Recently, authors began to generalize the uniform bounds in the curvature dimension condition to obtain Bonnet-Myers type theorems under almost positivity assumptions with exceptions \cite{liu2017distance,Muench-18}. A big gap in dealing with almost positive curvature is that it is not clear that Sobolev inequalities hold on graphs even in the case of uniformly positive curvature, such that techniques based on heat kernel estimates do not carry over immediately. To overcome this issue, we deal here with spectrally positive Ricci curvature in the sense that for some $K\geq 0$, we have 
\begin{equation}\label{specpos}
\frac 12\cL +\rho\geq K, 
\end{equation}
where $\cL=-\Delta\geq 0$ is the non-negative Laplacian of the graph and $\rho\colon V\to\R$ maps every point to the smallest value of the Bakry-{\'E}mery Ricci curvature at a given vertex. This will be made precise in the next section. 

Spectrally positive Ricci curvature on Riemannian manifolds proved to be a fruitful analytic generalization for most of the well-known results depending on uniformly bounded Ricci curvature or smallness of its negative part in $L^p$-sense, see \cite{CarronRose-18}. A very powerful assumption implying spectral positivity is the so-called Kato condition \cite{Rose-16a,Carron-16,Rose-17,RoseStollmann-18,Rose-19}. Denote for $x\in\R$ its negative part by $x_-:=-\min\{0,x\}$. The function $W\geq 0$ satisfies the Kato condition if there is a $T>0$ such that
\begin{equation}\label{kato}
 k_T(W):=\int_0^T \Vert P_tW\Vert_\infty \drm t <1,
\end{equation}
where $(P_t)_{t\geq 0}$ denotes the heat semigroup. The major advantage of this condition is that one can treat mapping properties of perturbed semigroups without knowing anything about the perturbed heat kernel itself, \cite{RoseStollmann-15}, allowing to obtain explicit results as soon as one knows something about the unperturbed heat kernel. Here, we use this condition for the function $W:=(\rho-K)_-$ for some $K>0$. 

In this paper, we will prove Lichnerowicz estimates for Bakry-{\'E}mery spectrally positive Ricci curvature, and we prove a Bonnet-Myers type theorem under a Kato condition on the curvature. Therefore, we use a combination of the original proof of Lichnerowicz together with the Perron-Frobenius theorem. Also under  spectrally positive Ricci curvature, we prove finiteness of the fundamental group. Assuming additionally that a refinement of \eqref{kato} holds, we prove a gradient estimate for the perturbed semigroup generated by $-\Delta+2\rho$, and use perturbation results to get elliptic Harnack and Buser inequalities. 
Originally, Harnack and Buser inequality are stated for non-negatively curvature. Our version requires a Kato condition comparing to positive curvature as some positive curvature is required to compensate the admissible occurrence of negative curvature.

\section{The setup}

\subsection{Graph Laplacians and variable Bakry-\'Emery curvature}

A measured and weighted graph (short \emph{mwg}) $G=(V,w,m)$ is a triple consisting of a countable set $V$, a symmetric function $w\colon V\times V\to [0,\infty)$ which is zero on the diagonal, and a function $m\colon V\to (0,\infty)$. We always assume that mwg are locally finite, i.e., given $G=(V,w,m)$ and $x\in V$, we always assume that the set of $y\in V$ with $w(x,y)>0$ is finite, and that $G$ is connected. We abbreviate $q(x,y):= w(x,y)/m(x)$, and define $\Deg(x):= \sum_{y\in V} q(x,y)$, $x\in V$. Furthermore, we let $\D:=\sup_x \Deg(x)$. 

\begin{definition}
Let $G=(V,w,m)$ be a mwg. We define the (combinatorial) metric $d\colon V\times V\to [0,\infty]$ by
\[
d(x,y):=\min\{n\in\N_0\mid \text{there exist }x=x_0,\dots, x_n=y \text{ s.t. }w(x_i,x_{i-1})>0\text{ for all }i=1,\ldots,n\},
\]
and the (combinatorial) diameter by $\diam G := \sup_{x,y\in V} d(x,y)$. 
\end{definition}
By $\cC(V)$ we denote the set of functions on $V$, and by $\cC_c(V)$ the functions in $\cC(V)$ with finite support. The map $\langle\cdot,\cdot\rangle\colon L^2(V)\times L^2(V)\to\R$ denotes the $L^2$-inner product and $\Vert\cdot\Vert_2^2:=\langle\cdot,\cdot\rangle$ the corresponding $L^2$-norm. The Laplacian $\Delta\colon\cC(V)\to\cC(V)$ is given by 
$$\Delta f(x):=\sum_{y\in V} q(x,y)(f(y)-f(x)), \quad x\in V, f\in \cC(V).$$
Note that $\Delta\leq 0$ in quadratic form sense. We denote by $L$ the non-negative Laplacian $L=-\Delta\geq 0$, which is the generator of the Dirichlet form
$$ Q(f,f)= \langle Lf,f\rangle,$$
and domain 
$$D(Q)=\overline{\cC_c(V)}\cap\{f\in L^2(V)\mid \Delta f\in L^2(V)\}$$
where the closure is taken w.~r.~t. the form norm $\Vert\cdot\Vert_Q^2:=Q(\cdot,\cdot)+\Vert\cdot\Vert_2^2$.

The definition of the Laplacian leads to the so-called carr\'e du champ operator $\Gamma$: for all $f,g\in\cC(V)$ bounded, $x\in V$:

$$\Gamma(f,g)(x) = \frac 12 (\Delta(fg)-f\Delta g -g\Delta f)(x).$$

Replacing the pointwise product by $\Gamma$, we can define another form $\Gamma_2,$ given by

$$\Gamma_2(f,g)(x)=\frac 12(\Delta\Gamma(f,g)-\Gamma(f,\Delta g)-\Gamma(g,\Delta f))(x),\quad x\in V, f,g\in\cC(V) \, \text{bounded}.$$

For simplicity, we abbreviate $\Gamma f:=\Gamma(f,f)$ as well as $\Gamma_2 f:=\Gamma_2(f,f)$. 

\begin{definition}
Let $G=(V,w,m)$ be a mwg and $\rho\colon V\to\R$ and $n\in(0,\infty]$. We say that $G$ satisfies the variable curvature dimension condition $CD(\rho,n)$ iff for all bounded $f\in\cC(V)$, we have
 $$\Gamma_2(f)(x)\geq \rho(x)\Gamma(f)(x)+\frac 1n (\Delta f)^2(x), \quad  x\in V.$$
\end{definition}

%
%
%
%
%
%
%
%
%
%
%
%


%
%
%
%
%
%
%
%
%
%
%
%


\subsection{Spectral curvature assumptions and the Kato condition}

\begin{definition}
Let $G=(V,w,m)$ be a mwg, $\rho\in\cC(V)$, $n\in (0,\infty]$, and $K>0$. We say that $G$ has $(K,n,\rho)$-spectrally positive Ricci curvature iff $CD(\rho,n)$ holds and 
 $$\frac 12 L +\rho\geq K$$
 in quadratic form sense, i.e.,
 $$\frac 12 \langle Lf,f\rangle + \langle \rho f,f\rangle\geq K\langle f,f\rangle,\quad f\in \cC_c(V).$$

\end{definition}

Let $(P_t)_{t\geq 0}:=(\euler^{-tL})_{t\geq 0}$ denote the heat semigroup generated by $L$, and for a function $W\colon V\to\R$ we abbreviate $$(P_t^W)_{t\geq 0}:=(\euler^{-t(L+W)})_{t\geq 0}.$$ 

Since $L$ generates a Dirichlet form in $L^2(V)$, we can consider the following quantity which proved to be useful in perturbation theory of Dirichlet forms: 
For $T>0$ and $W\colon V\to [0,\infty)$, we let
$$k_T(W):=\sup_{n\in\N}\int_0^T \Vert P_t(n\wedge W)\Vert_\infty \drm t,$$
and say that $W$ satisfies the \emph{Kato condition} if there is a $T>0$ such that $k_T(W)<1$. In this case, general results from \cite{StollmannVoigt-96,Voigt-86} imply that mapping properties from $(P_t)_{t\geq 0}$ carry over to $(P_t^W)_{t\geq 0}$ and can also be improved. (Note that the truncation procedure ensures that $k_T(W)$ is well-defined.) 
\begin{remark}
The Kato condition proved to be very useful in the context of Riemannian manifolds, see, e.g., \cite{Carron-16, Rose-16a, CarronRose-18, RoseStollmann-18,Rose-19} and the references therein.
\end{remark}

To illustrate the meaning of the Kato condition on graphs, we construct a graph satisfying the Kato condition with negative curvature somewhere.
The graph consists of three vertices $V=\{x,y,z\}$ with $x \sim y \sim z \not \sim x$ such that $m(x) = \eps$ and $m(y)=1$ and $w(x,y)=1$.
Regarding the curvature computation, we refer to
\cite[Proposition~2.1]{hua2017ricci}.
We choose $w(y,z)$ such that the curvature at $x$ is a bit negative, e.g. $w(y,z)= \frac 1 \eps + 3 + \eps$ and now we can choose $m(z)$ such that the curvature at $y$ and $z$ is very positive, e.g., $m(z)=\eps$. It is straight forward to check that this graph satisfies the Kato condition from the following Lemma for $K=T=1$ and sufficiently small $\eps$.




\begin{lemma}\label{sglemma}
Let $G$ be a finite or infinite mwg, and $K,T>0$, $n\in(0,\infty]$, $\rho\in\cC(V)$. If $CD(\rho,n)$ holds and
\[
k_T((\rho-K)_-)\leq\frac 12\left(1-\euler^{-KT/4}\right),
\]

then $G$ is $(K/2,n,\rho)$-spectrally positive.
Further, the semigroup $P_t^{\rho}$ maps $L^p(V)$ to $L^p(V)$ continuously for $t>0$ and all $p\in [1,\infty]$, and we have 
$$\Vert P_t^{\rho}\Vert_{\infty,\infty}\leq \euler^{K(T-t)/2}, \quad t>0.$$
Moreover, 
\[
k_T((\rho-K)_-)\leq\frac 14\left(1-\euler^{-KT/2}\right)
\]
implies
$$\Vert P_t^{2\rho}\Vert_{\infty,\infty}\leq \euler^{K(T-t)}, \quad t>0.$$
\end{lemma}

Let us briefly give an intuition of the assumption of the Kato constant in the lemma. The parameter $K$ tells up to which threshold positive curvature is employed to compensate negative curvature. The time $T$ tells how thick the spikes of negative curvature can be. Note that there is only some compact time window away from zero in which $T$ can be chosen provided that the curvature bound $CD(K,\infty)$ is violated.

We remark that the first part of the lemma is its classical version adapted from the Dirichlet form setting. The second part is a straight forward renormalization which will be applied in our curvature results. For convenience of the reader, we state the second part explicitly.

\begin{proof}
Since $L$ generates a Dirichlet form, \cite[Theorem~3.1]{StollmannVoigt-96} implies that $(\rho-K)_-$ is relatively bounded with respect to $L$. Furthermore, as explained after \cite[Remark~2.3]{CarronRose-18}, a simple calculation shows that
$$\frac 12 L+\rho\geq K/2,$$ giving that $G$ is $(K/2,n,\rho)$-spectrally positive. 

The continuity properties for the perturbed semigroup follow directly from \cite[Theorem~3.1]{StollmannVoigt-96}. The Myadera-Voigt perturbation theorem, see \cite[Theorem~2.1]{Voigt-86}, gives
$$\Vert P_t^{-(\rho-K)_-}\Vert_{\infty,\infty}\leq \euler^{KT/2}, \quad t\in(0,T].$$
Hence,
\[
\Vert P_t^{\rho}\Vert_{\infty,\infty}=\Vert P_t^{K+\rho-K}\Vert_{\infty,\infty}\leq \euler^{-Kt}\Vert P_t^{-(\rho-K)_-}\Vert_{\infty,\infty}\leq \euler^{-Kt}\euler^{KT/2}, \quad t\in(0,T].
\]
For $t>T$, we split $t=kT+\epsilon$, $k\in\N$, $\epsilon\in(0,T]$, and calculate
\[
\Vert P_t^{\rho}\Vert_{\infty,\infty}\leq\Vert P_{kT}^{\rho}\Vert_{\infty,\infty}\Vert P_\epsilon^{\rho}\Vert_{\infty,\infty}\leq \Vert P_{T}^{\rho}\Vert_{\infty,\infty}^k\Vert P_\epsilon^{\rho}\Vert_{\infty,\infty}\leq \euler^{-KTk}\euler^{kKT/2}\euler^{-K\epsilon}\euler^{KT/2}\leq \euler^{-Kt/2}\euler^{KT/2}.
\]
For the second part of the lemma, note that 
\[
\kappa_T((\rho-K)_-)\leq \frac 14\left(1-\euler^{-KT/2}\right)
\]
implies
\[
\kappa_T((2\rho-2K)_-)\leq \frac 12\left(1-\euler^{-2KT/4}\right),
\]
leading to 
\[
\Vert P_t^{2\rho}\Vert_{\infty,\infty}\leq \euler^{ KT}\euler^{- Kt},\quad t>0.
\]
from the first part of the lemma applied to $2\rho$ and $2K$.
\end{proof}

\section{Lichnerowicz estimate for spectrally positive curvature}

We first prove a Lichnerowicz-type theorem assuming spectrally positive Ricci curvature since it does not depend on the Kato condition. 

\begin{theorem}

Let $K>0$, $\rho\in\cC(V)$, $n\in(0,\infty]$, and $G$ a $(K,n,\rho)$-spectrally positive finite mwg. Then, we have $\lambda(G) \geq K$, where $\lambda(G)$ is the first positive eigenvalue of $L$ in $L^2(V)$.
\end{theorem}

\begin{proof}

Let $f$ be an eigenfunction to the eigenvalue $\lambda$ of $L$, i.e.,
 $Lf=\lambda f$ and let $\phi \geq 0$. Then,
\begin{align*}
\frac 1 2 \langle\Delta \phi,\Gamma f\rangle + \lambda\langle \phi,\Gamma f\rangle
=\frac 1 2 \langle\phi, \Delta  \Gamma f\rangle - \langle \phi,\Gamma (f,\Delta f)\rangle
=
\langle \Gamma_2 f,\phi \rangle \geq \langle\Gamma f,\rho \phi \rangle  
\end{align*}
and thus,
\[
\lambda\langle \phi,\Gamma f\rangle \geq \langle (L/2 + \rho)\phi,\Gamma f\rangle.
\]
Let $E$ be the smallest eigenvalue of $L/2+\rho$,
i.e.,
\[
E= \inf_{h \in C(V)} \frac{\langle(L/2 + \rho)h,h \rangle}{\|h\|_2^2}.
\]
Then, $E \geq K$.
By the Perron-Frobenius theorem, there exists $\phi :V \to (0,\infty)$ s.t.
\[
(L/2+\rho - E) \phi =0.
\]
Putting together yields
\[
\lambda\langle \phi,\Gamma f\rangle \geq \langle (L/2 + \rho)\phi,\Gamma f\rangle = E \langle \phi, \Gamma f \rangle.
\]
Since $\phi>0$ everywhere and $\Gamma f >0$ somewhere, we have $\lambda \geq E \geq K$. This finishes the proof.
\end{proof}

\section{Gradient estimates and Harnack inequality}

A modification of the classical Ledoux ansatz leads to the following gradient estimate.
\begin{theorem}\label{thm:GradEst}

Let $G$ be a mwg with $\D<\infty$ satisfying $CD(\rho,n)$ for $\rho\colon V\to \R$ and $n\in(0,\infty]$. Then, for all bounded $f\in\cC(V)$, we have the pointwise inequalities
\begin{enumerate}[(i)]
\item
\begin{equation}\label{easyestimate}
\Gamma P_tf\leq P_t^{2\rho}\Gamma f,
\end{equation}
\item
\begin{equation}\label{gradientestimate}
2t\Gamma P_t f\leq P_t^{-2\rho_-} f^2  - (P_t f)^2,
\end{equation}
\item and assuming $\rho\leq K$, we get  
\begin{equation}\label{goodestimate} 
\frac 1{2K}(\euler^{2Kt}-1)\Gamma P_tf\leq \euler^{2Kt}P_t^{2\rho}f^2-(P_tf)^2.
\end{equation}
\end{enumerate}
\end{theorem}
\begin{proof}
As usual, we omit the dependency of the vertex. Define for fixed $f\in\cC(V)$
\[
H(s):=P_s^{2\rho} \Gamma P_{t-s} f.
\]
Then, following the classical Ledoux ansatz, we have
$H'(s)\geq 0$, where we used the fact that the vertex degree is uniformly bounded as in \cite[Prop.~2.1]{lin2015equivalent}. Integration w.r.t. $t$ gives \eqref{easyestimate}.
On the other hand,
%
the function
\[
F(s)= P_s^{-2\rho_-}(P_{t-s}f)^2
\]
satisfies $F'\geq 2H$. Hence,
\[
F(t)-F(0) = \int_0^t F'(s)ds \geq \int_0^t 2H(s) ds \geq 2tH(0).
\]
Observing that
\[
F(0)=(P_t f)^2 
\]
and
\[
F(t)= P_t^{-2\rho_-} f^2
\]
leads to \eqref{gradientestimate}.
For \eqref{goodestimate}, we define
\[
\phi:= P_s^{2\rho}(P_{t-s}f)^2.
\]
Then, using \eqref{easyestimate},
\[
\phi '(s)\geq -2K\phi(s)+2\Gamma P_tf.
\]
Hence,
\[
(\euler^{2Ks}\phi(s))'\geq \euler^{2Ks}\Gamma P_tf.
\]
Integration leads to the desired etimate.
\end{proof}

Having the Kato condition, we can estimate $\Gamma P_t f$ in terms of $\|f\|_\infty$.

\begin{corollary}\label{cor:GradientVsInfinityNorm}
Suppose $K,T>0$, $n\in(0,\infty]$, $\rho\in\cC(V)$, $G$ a mwg with $\D<\infty$ satisfying $CD(\rho,n)$, and
\[
k_T((\rho-K)_-)\leq\frac 14\left(1-\euler^{-KT/2}\right).
\]
Then for all $f \in\ell_{\infty}(V) $ and all $t>0$,
\[
\Gamma P_t f \leq Ke^{KT}\frac{1}{\sinh (Kt)} \|f\|_\infty^2 \leq \frac{e^{KT}}t \|f\|_\infty^2.
\]
\end{corollary}

\begin{proof}
Without loss of generality, we can assume $\rho \leq K$.
By Theorem~\ref{thm:GradEst}, we have
\[
\Gamma P_t f \leq \frac{2K}{e^{2Kt} - 1} e^{2Kt} P_t^{2\rho} f^2.
\]
By Lemma~\ref{sglemma}, we have
\[
P_t^{2\rho} f^2 \leq e^{K(T-t)}\|f\|_\infty^2.
\]
Putting together gives
\[
\Gamma P_t f \leq Ke^{KT}\frac{2e^{Kt}}{e^{2Kt}-1}\|f\|_\infty^2 = Ke^{KT}\frac{1}{\sinh (Kt)}\|f\|_\infty^2 \leq \frac{e^{KT}}t \|f\|_\infty^2.
\]
This finishes the proof.
\end{proof}

The following corollary generalizes the Harnack inequality from \cite{chung2014harnack} to non-constant curvature bounds.

\begin{corollary}
Suppose $K,T>0$, $n\in(0,\infty]$, $\rho\in\cC(V)$, $G$ a finite mwg satisfying $CD(\rho,n)$, and
\[
k_T((\rho-K)_-)\leq\frac 14\left(1-\euler^{-KT/2}\right).
\]

Then, for any eigenfunction $f$ with eigenvalue $\lambda$ of $L$ in $L^2(G)$ we have, 
\[
\Vert \Gamma f\Vert_\infty\leq {2 \euler^{1+TK}} \cdot \lambda \Vert f\Vert_\infty^2.
\]
\end{corollary}

\begin{proof}
Let $f$ be an eigenfunction of $L$ with eigenvalue $\lambda$. We have for any $t>0$, by Corollary~\ref{cor:GradientVsInfinityNorm},
\[
\euler^{-2t\lambda}\Gamma f= \Gamma P_tf\leq \frac{\euler^{KT}}{t}\Vert f\Vert_\infty^2.
\]
Choosing $t=(2\lambda)^{-1}$ gives the result.
\end{proof}

\section{Diameter bounds and curvature in Kato class}

Similarly to \cite{liu2018bakry,
liu2017distance}, we apply a gradient estimate (Theorem~\ref{thm:GradEst}) to conclude a diameter bound.

\begin{theorem}
Let $T,K>0$, $n\in(0,\infty]$, $\rho\in\cC(V)$, $G$ a mwg with $\D<\infty$ and $CD(\rho,n)$. Suppose
\[
k_T((\rho-K)_-)\leq\frac 14\left(1-\euler^{-KT/2}\right).
\]
Then,
\[
\diam(G) \leq 4\D\euler^{KT/2}/K.
\]
\end{theorem}

\begin{proof}
We follow the proof of \cite[Theorem~2.1]{liu2018bakry}. First, note that for any $g\in\cC(V)$, we have $(\Delta g(x))^2\leq 2\Deg(x)\Gamma g (x)$, $x\in V$. Fix $y\in V$, $R>0$, and define $f(x):= (d(x,y)-R)_+$. Thanks to $\Gamma f\leq 1$, and \eqref{easyestimate}, we have for all $t>0$ and all $x \in V$,

\[
\vert \partial_t P_tf(x)\vert^2 =\vert \Delta P_t f(x)\vert^2\leq 2\Deg(x)\Gamma P_tf(x)\leq 2\D P^{2\rho}_t\Gamma f(x) \leq \D^2\euler^{KT}\euler^{-Kt}
\]
where the last inequality  follows from Lemma~\ref{sglemma}.

Therefore, 
\begin{align*}
\vert P_t f- f\vert\leq \int_0^t\vert \partial_s P_s f\vert \drm s \leq \int_0^\infty\vert \partial_s P_s f\vert \drm s\leq 2\D\euler^{KT/2}/K.
\end{align*}

The triangle inequality yields 
\[
d(x,y)\leq \vert f(x)-f(y)\vert\leq \vert P_tf(x)-f(x)\vert +\vert P_tf(y)-f(y)\vert +\vert P_tf(x)-P_tf(y)\vert\to 4\D\euler^{KT/2}/K, \quad t\to \infty.
\]
This finishes the proof as $x,y$ are chosen arbitrarily. 
\end{proof}

\section{Fundamental group and spectrally positive curvature}
It has been shown in \cite{kempton2017relationships} that strictly positive Bakry Emery curvature everywhere implies finiteness of the fundamental group.
By the fundamental group we refer to the $CW$-complex consisting of the graph as $1$-skeleton where all cycles of length 3 and 4 are filled with disks.
In order to obtain finiteness of the fundamental group, it turns out to be sufficient to assume spectrally positive curvature which is significantly weaker than the  Kato condition or a uniform positive lower curvature bound.

\begin{theorem}\label{thm:FundamentalGroup}
Let $G$ be a spectrally positive finite mwg. Then, the fundamental group is finite, i.e., the universal cover of the $CW$-complex explained above is finite.
\end{theorem} 

\begin{proof}
Since $G$ is spectrally positive, it satisfies $CD(\rho,\infty)$ with $\frac 1 2 L +\rho >0$ for some $\rho:V \to \R$.
Let $\widetilde G$ be the 1-skeleton of the universal cover with edge weights and vertex measure inherited from $G$. Let $\phi: \widetilde G \to G$ be the corresponding covering map.
Since the Bakry-\'Emery curvature of a vertex only depends on its second neighborhood, we have that
$\widetilde G$ satisfies $CD( \rho \circ \phi,\infty)$.
Let $\eta:V \to \R$ be the eigenfunction to the smallest eigenvalue $K>0$ of $\frac 1 2 L + \rho$.
Then, $\eta$ is positive by the Perron Frobenius theorem.
Moreover, $\eta \circ \phi$ is a positive eigenfunction of $\frac 1 2 L_{\widetilde G} + \rho \circ \phi$.
Let $\widetilde x$ be a vertex in $\widetilde G$ and let $f := d(\widetilde x, \cdot) \wedge C$ for some large $C$ we will choose later. Then by Theorem~\ref{thm:GradEst}, we have
\[
\Gamma P_t f \leq P_t^{2\rho\circ \phi} \Gamma f \leq \frac 1 2 \D P_t^{2\rho\circ \phi} 1
\]
where $\Gamma$ and $P_t$ are the operators acting on $\widetilde G$.
Moreover, we have
\[
(\min \eta ) \cdot P_t^{2\rho \circ \phi} 1 \leq P_t^{2\rho \circ \phi} \eta\circ \phi \leq \euler^{-2Kt}   \eta\circ \phi  \leq \euler^{-2Kt} \max \eta
\]
and thus,
\[
(\partial_t P_t f)^2=(\Delta P_t f)^2 \leq 2\D \Gamma P_t f \leq \D^2 \cdot \frac{\max \eta}{\min \eta} \cdot \euler^{-2Kt}
\]
Taking square root and integrating gives
\[
|\lim_{t\to \infty} P_t f - f| \leq \frac{\D}K \sqrt{\frac{\max \eta}{\min \eta}}
\]
As $P_t$ converges to a constant in space as $t \to \infty$, we obtain for all vertices $\widetilde y$ in $\widetilde G$,
\[
|f(\widetilde x) - f(\widetilde y)| \leq \frac{2\D}K \sqrt{\frac{\max \eta}{\min \eta}}.
\]
When choosing the cutoff constant $C$ larger than the right hand side, this shows that the diameter of $\widetilde G$ is finite as $\widetilde x$ and $\widetilde y$ are chosen arbitrarily.
This proves that the universal cover and thus the fundamental group are finite.
\end{proof}

We now apply the theorem to graphs with non-negative curvature everywhere and positive curvature somewhere. The following corollary answers a question in \cite{kempton2017relationships} subsequent to Corollary~6.4 therein.

\begin{corollary}
Let $G$ be a connected graph with non-negative curvature everywhere and positive curvature at some vertex.
Then, the fundamental group is finite.
\end{corollary}
\begin{proof}
By assumption, $G$ satisfies $CD(\rho,\infty)$ with $\rho \geq 0$ and $\rho(x) >0$ for some vertex $x$. Since $G$ is connected, this yields
\[
\frac 1 2 L + \rho >0.
\]
Applying Theorem~\ref{thm:FundamentalGroup} proves the corollary.
\end{proof}

\section{Buser's inequality and curvature in Kato class}
We prove a Buser inequality under our generalized curvature conditions. Therefore, we need to introduce the Cheeger constant of a finite weighted graph $G$, that is given by 

$$h(G):= \inf\left\{\frac{\vert \partial U\vert}{\vol U}\mid \emptyset \neq U\subset V, \vol U\leq \frac 12 \vol V\right\},$$

where $\vert \partial U\vert := \sum_{x\in U, y\in V\setminus U} w(x,y)$ and $\vol U=\sum_{x\in U}m(x)$. 

Buser's inequality gives the counterpart to the well known Cheeger inequality stating that the first positive eigenvalue is larger than the square of the Cheeger constant $h$ up to a constant factor. Buser inequality is the reverse inequality which has been proven classically under non-negative Ricci curvature \cite{buser1982note}.
Recently, there has been significant interest to obtain analogous results in the discrete setting (see \cite{liu2019curvature,liu2018eigenvalue,
liu2019buser,klartag2016discrete} for Bakry Emery curvature, \cite{munch2019non} for Ollivier curvature, and \cite{erbar2018poincare} for entropic curvature)

\begin{proposition}

Suppose $K,T>0$,  $\rho\in\cC(V)$, $G$ a mwg with $\D<\infty$ and $CD(\rho,\infty)$, and
\[
k_T((\rho-K)_-)<\frac 14\left(1-\euler^{-KT/2}\right).
\] 
Then, there is a constant  $c=c(K,T)>0$ such that

\[
\Vert f-P_tf\Vert_1\leq c \,\sqrt t\, \Vert \sqrt{\Gamma f}\Vert_1.
\]
\end{proposition}

\begin{proof}
The proof is essentially the same as in \cite{klartag2016discrete}. Nevertheless, we give a short outline to calculate the constants. Again, w.l.o.g., we assume $\rho\leq K$. Since
\[
f-P_tf=\int_0^t \Delta P_sf\drm s,
\]
we have for any $g\in L^\infty$ by Corollary~\ref{cor:GradientVsInfinityNorm},
\begin{align*}
\sum m(x)g(x)( f(x)-P_sf(x)) &=\int_0^t \sum m(x) \Gamma(P_s g, f)(x)\drm s
\leq \int_0^t\sum m(x)\sqrt{\Gamma(P_sg)(x)}\sqrt{\Gamma f(x)}\drm s\\
&\leq \int_0^t\Vert \sqrt{\Gamma f}\Vert_1 \sqrt{\Vert\Gamma{P_sg}\Vert_\infty}\drm s
\leq \Vert \sqrt{\Gamma f}\Vert_1\int_0^t \left(\frac {e^{KT}}{s}\Vert g\Vert_\infty^2\right)^\frac{1}{2}\drm s\\
&=  2\euler^{KT} \Vert \sqrt{\Gamma f}\Vert_1 \Vert g\Vert_\infty \sqrt{t} .
\end{align*}
Choosing $g = \operatorname{sgn} (f-P_t f)$  leads to the result.
\end{proof}

\begin{theorem}

Suppose $K,T>0$, $\rho\in\cC(V)$, $G$ a mwg with $\D<\infty$ and $CD(\rho,\infty)$, and
\[
k_T((\rho-K)_-)<\frac 14\left(1-\euler^{-KT/2}\right).
\] 
Then, there exists a constant  $c>0$ depending only on $K,T$ and on $\inf_{x\sim y} \frac{w(x,y)}{m(x)}$ such that
\[
\lambda_1\leq c h^2.
\]
\end{theorem}

\begin{proof} The proof is exactly the same as that of \cite[Theorem~4.2]{klartag2016discrete} apart from the different constant given by the above proposition.
\end{proof}

\bibliographystyle{alpha}

\end{document}